\documentclass[a4paper,12pt]{amsart}
\usepackage{amssymb,amsfonts,amsmath,amsthm,amscd}
\usepackage[all,cmtip]{xy}
\usepackage{makecell}
\usepackage{multirow}
\usepackage{color}

\newtheorem{theorem}{Theorem}[section]
\newtheorem{lemma}[theorem]{Lemma}
\newtheorem{corollary}[theorem]{Corollary}
\theoremstyle{definition}

\newtheorem{example}[theorem]{Example}

\theoremstyle{remark}

\numberwithin{equation}{section}

\begin{document}
\title[Kernels of  linear operators]{The kernels of powers of linear operator via Weyr characteristic}

\author{Jie Jian}
\address{School of Mathematics and Statistics,
Hubei University,
Wuhan, $430062$, P. R. China}
\curraddr{}
\email{jjian22@163.com}
\thanks{}

\author{Jun Liao}
\address{School of Mathematics and Statistics,
Hubei University,
Wuhan, $430062$, P. R. China}
\curraddr{}
\email{jliao@hubu.edu.cn}
\thanks{}

\author{Heguo Liu}
\address{Department of Mathematics, Hainan University, Haikou, $570228$, China
}
\curraddr{}
\email{ghliu@hainanu.edu.cn}
\thanks{}

\date{}

\dedicatory{}

\begin{abstract}
The adjoint of a matrix in the Lie algebra associated with a matrix algebra is a fundamental operator, which can be generalized to a more general operator $\varphi_{AB}: X\rightarrow AX-XB$ by two matrices $A$ and $B$. The kernel of the operator is very well-known and it can be found in Gantmacher's book. The formulas for the dimensions of the kernels of arbitrary powers of the operator $\varphi_{AB}$ were given in terms of the Segre characteristics of these two matrices by the second and third authors in this paper and their collaborators. This paper provides an alternative approach to this problem via the Weyr characteristic in a more essential method. We obtain formulas for the dimensions of the kernels of arbitrary powers of the operator in terms of the Weyr characteristics. Furthermore, the basis for kernel of each power of the operator is described explicitly. As a consequence, for arbitrary square matrices $A$ and $B$ over an algebraically closed field, the dimension of the kernel of each power of the operator $\varphi_{A-\lambda I,B}$ for eigenvalues $\lambda$ of $\varphi_{AB}$ can be viewed as a similarity invariant of the operator $\varphi_{AB}$, so we characterise the operator within similarity, which should be of interest to a number of people (including physicists).
\end{abstract}

\maketitle

\section{Introduction}\label{section1}

Let $M_{n\times n}(\mathbb{C})$ be the algebra of $n\times n$ matrices over the complex field $\mathbb{C}$. Then $M_{n\times n}(\mathbb{C})$ becomes a Lie algebra with the Lie product $[A, B]=AB-BA$. The adjoint $\operatorname{ad} A(X)=[A, X]$ is an operator of $M_{n\times n}(\mathbb{C})$. The formula for the dimension of the kernel space of $\operatorname{ad} A$ is due to Frobenius, see \cite[Theorem VII.1]{Wedd} or  \cite[Chapter VIII, Theorem 2]{Gantmacher}. 
Subsequently, Gracia obtained the formulas for the dimensions of $\operatorname{ker}(\operatorname{ad}^{2}A)$ and $\operatorname{ker}(\operatorname{ad}^{3}A)$ in \cite{Gra}. Note that the kernels of $\operatorname{ad}^{2}A$ and $\operatorname{ad}^{3}A$ are the solutions of the matrix equations $[A,[A,X]]=0$ and $[A,[A,[A,X]]]=0$ of the second and third fold Lie product, respectively.

Suppose $A\in M_{m\times m}(\mathbb{C})$ and $B\in M_{n\times n}(\mathbb{C})$. Let $\varphi_{AB}$ be the operator of $M_{m\times n}(\mathbb{C})$ defined by $\varphi_{AB}(X)=AX-XB$ for $X\in M_{m\times n}(\mathbb{C})$.
The kernel $\operatorname{ker} (\varphi_{AB}^{})$ of the operator $\varphi_{AB}$ is very well-known and it can be found  in Gantmacher's book \cite[Chapter VIII, Theorem 1]{Gantmacher}.
The formulas for the dimensions of $\operatorname{ker} (\varphi_{AB}^{2})$ and $\operatorname{ker}(\varphi_{AB}^{3})$ were determined  in \cite{LLW}. The formulas for the dimensions of the kernels of arbitrary powers of the operator were given in terms of the Segre characteristics of $A$ and $B$ by the second and third authors in this paper and their collaborators in \cite{LLWX}. In particular, the general solution of $[A,...,[A,X]]=0$ of the $k$-fold Lie product with $k > 1$ was determined. 

It is well known that the Jordan and Weyr canonical forms are different canonical forms for similarity. One of the differences between them is that the Jordan canonical form works well for problems involving powers of matrices, while the Weyr canonical form works well for problems involving commutativity and interactions involving more than one matrix. The Weyr form did not become popular among
mathematicians and it was overshadowed by the closely related, but
distinct, canonical form known by the name Jordan canonical form.
Recently several applications have been found for the Weyr matrix, 
of particular interest is an application of the Weyr matrix in the study of
phylogenetic invariants in biomathematics, for more details of the Weyr form we refer to \cite{MCV} and  \cite[Section 3.4, Notes and Further Readings]{HJ}. Please refer to \cite[Section 2.4, Historical Remarks]{MCV} for historical remarks. There have been several relatively recent papers involving the Weyr form, for example, 
Bondarenko, Futorny, Petravchuk and Sergeichuk gave a simple normal form under similarity of a pair of commuting nilpotent matrices using the Weyr form in \cite[Theorem 2]{BFPS}.
Moreover, they gave a method for constructing all matrices commuting with a given Weyr matrix in \cite[Theorem 3]{BFPS}.
In the study of commutative finite-dimensional algebras in \cite{MW}, O’Meara and Watanabe made use of the Weyr structures for block matrices in order to apply to the monomial complete intersection ring and the Hilbert function of an artinian algebra. 
Cruz and Misa in \cite{CM} and Cruz and Tanedo in \cite{CT} also took advantage of the Weyr form in dealing with problems concerning matrix centralizers.
Holbrook and O’Meara in \cite{HO} developed a computing strategy to resolve the Gerstenhaber problem by constructing counter-example, the key tool is again the Weyr form.
The Weyr form has been shown to be a better tool than its Jordan cousin in a number of situations, all the instances above show that the Weyr form seems better suited than its Jordan counterpart.

Here we collect some related comments about the centralizers of Weyr blocks which are taken from \cite{MCV, HJ}.
The Weyr form (in its standard partition) was rediscovered by Belitskii, whose motivation was to find a canonical form for similarity with the property that every matrix commutes with it is block upper triangular. 
Belitskii also described identities among the blocks of the standard partition of a matrix that commutes with a Weyr block. Belitskii was probably the first to observe the nice block upper triangular form of the  matrices that centralize a given nilpotent matrix in Weyr form, see \cite[Proposition 2.3.3, Proposition 3.2.2]{MCV} or \cite[Theorem 3.4.2.10a]{HJ}. 
The remarkable theorem \cite[Theorem 3.4.2.10b]{HJ} of O’Meara-Vinsonhaler about commuting families contains yet another rediscovery of the Weyr canonical form as well as the efficient formulation \cite[Equation 3.4.2.12]{HJ} of the identities among the blocks of matrices that commute with a Weyr block, see also Lemma \ref{kernal-image} below for intertwining relation of two Weyr blocks.

Essentially, $\operatorname{ker} \varphi_{AB}^{\ell}$ is mainly about commutativity in some sense, so it is better to use the Weyr canonical form to calculate $\operatorname{ker} \varphi_{AB}^{\ell}$. It is the proof of the main theorem that requires moving to the associated Weyr forms.
At the same time, using the Weyr canonical form has its own particular interest for matrix computations and potential applications.
In fact, the referee of \cite{LLWX} encouraged the authors to try to express the dimension in terms of the Weyr characteristics, since $\operatorname{dim}\operatorname{ker} \varphi_{AB}^{\ell}$ was given in terms of the Segre characteristics in \cite{LLWX}, while the Weyr characteristic was desirable in this situation.
In this paper, we obtain the formula for the dimension of $\operatorname{ker} \varphi_{AB}^{\ell}$ for every $\ell$ in terms of the characteristics of Weyr.
Furthermore, the basis of $\operatorname{ker} \varphi_{AB}^{\ell}$ is described explicitly. 
As a consequence, 
the dimensions of the kernels $\operatorname{ker} \varphi_{A-\lambda I,B}^{\ell}$ of powers of the operator $\varphi_{A-\lambda I,B}$ for eigenvalues $\lambda$ of $\varphi_{AB}$ can be viewed as a similarity invariant of the operator $\varphi_{AB}$. Therefore, we characterise the operator $\varphi_{AB}$ within similarity.

For convenience, we collect some definitions and  terminologies from \cite{HJ}.
Let $A\in M_{m\times m}(\mathbb{C})$, let $\lambda\in \mathbb{C}$, let $k$ be a positive integer, let
\[\operatorname{nullity}(A-\lambda I)^{k}=\operatorname{dim}\operatorname{ker}(A-\lambda I)^{k}, \; \operatorname{nullity}(A-\lambda I)^{0}=0\]
and define 
\[\omega_{k}(A,\lambda)=\operatorname{nullity}(A-\lambda I)^{k}-\operatorname{nullity}(A-\lambda I)^{k-1}\]
The index of an eigenvalue $\lambda$ of $A$ is the smallest positive integer $q$ such that $\operatorname{nullity}(A-\lambda I)^{q}=\operatorname{nullity}(A-\lambda I)^{q+1}$. 
The Weyr characteristic (or Weyr structure ) of $A\in M_{m\times m}(\mathbb{C})$ associated with $\lambda\in \mathbb{C}$ is
\[\omega(A,\lambda)=(\omega_{1}(A,\lambda),\omega_{2}(A,\lambda),\dots,\omega_{q}(A,\lambda))\]
in which $q$ is the index of the eigenvalue $\lambda$ of $A$. In particular, 
\[\omega_{1}(A,\lambda)\geq \omega_{2}(A,\lambda)\geq \dots \geq \omega_{q}(A,\lambda)\geq 1\] 
is a nonincreasing sequence of positive integers. Without causing confusion, we simply write $\omega_{i}=\omega_{i}(A,\lambda)$ and $\omega=(\omega_{1},\omega_{2},\cdots,\omega_{q})$. 
The Weyr block of $A$ associated with the eigenvalue $\lambda$ is the upper triangular $q\times q$ block bidiagonal matrix
$$W_{A}(\lambda)=\begin{pmatrix} 
	\lambda I_{\omega_{1}} & N_{\omega_{1},\omega_{2}} &   &  & \\ 
	& \lambda I_{\omega_{2}} &  N_{\omega_{2},\omega_{3}}  &  &\\    
	&  & \ddots  &\ddots   &\\  
	&   &   &\ddots &N_{\omega_{q-1},\omega_{q}} \\  
	&   &    & & \lambda I_{\omega_{q}} \\  
\end{pmatrix}=\lambda I_{\omega}+N_{\omega}   	$$
in which
$N_{\omega_{i},\omega_{j}}= \begin{pmatrix}
	I_{\omega_{j}}   \\  
	0\\   
\end{pmatrix}\in M_{\omega_{i}\times \omega_{j}}(\mathbb{C}),  i<j$,  and  $I_{\omega_{j}}$ is the identity $\omega_{j}\times \omega_{j}$ matrix.
Suppose that  $\lambda_{1},\lambda_{2},\dots,\lambda_{a}$ are distinct eigenvalues of $A\in M_{m\times m}(\mathbb{C})$ in any given order. The Weyr matrix $W_{A}=W_{A}(\lambda_{1})\oplus W_{A}(\lambda_{2})\oplus \dots\oplus W_{A}(\lambda_{a})$  is a Weyr canonical form of $A$, which is unique up to permutations of its direct summands.
Please note that the calculation of the Segre characteristic is not nearly as straightforward in terms of nullities of powers as for the Weyr characteristic, see \cite[Proposition 2.2.3 and Corollary 2.4.6]{MCV}, although the Jordan structure can be deduced as the dual partition of the Weyr structure,  see \cite[Theorem 2.4.1 and
Corollary 2.4.6]{MCV}.

By the Weyr canonical form theorem (also called the Weyrisimilitude theorem) \cite[Theorem 2.2.2]{MCV} or \cite[Theorem 3.4.2.3]{HJ}, 
every square matrix over an algebraically closed field is similar to a matrix in Weyr canonical form.
Suppose $A\in M_{m\times m}(\mathbb{C})$ and $B\in M_{n\times n}(\mathbb{C})$. Then there exist invertible matrices $P\in \operatorname{GL}_{m}(\mathbb{C})$ and $Q\in \operatorname{GL}_{n}(\mathbb{C})$
such that $P^{-1}AP=W_{A}$ and $Q^{-1}BQ=W_{B}$ are Weyr matrices, respectively.
Then for each $\ell$, we have $$P^{-1}\varphi_{AB}^{\ell}(X)Q=\varphi_{W_{A}W_{B}}^{\ell}(P^{-1}XQ)$$
Suppose that $\alpha_{i}=(\alpha_{i1},\alpha_{i2},\dots ,\alpha_{ip_i})=\omega(A,\lambda_{i})$ is the Weyr characteristic of $A$ associated with eigenvalue $\lambda_{i}$, for $i=1,2,\dots, a$,  and $\beta_{j}=(\beta_{j1},\beta_{j2},\dots ,\beta_{jq_j})=\omega(B,\mu_{j})$ is the Weyr characteristic of $B$ associated with eigenvalue $\mu_{j}$, for $ j=1,2,\dots,b$.
Then the Weyr canonical forms of $A$ and $B$ are respectively
$$W_{A}= \begin{pmatrix}
	\lambda_{1}I_{\alpha_{1}}+N_{\alpha_{1}} &   &   &   \\   
	&	\lambda_{2}I_{\alpha_{2}}+N_{\alpha_{2}}  &    &  \\  
	&   & \ddots  &  \\  
	&   &    &  \lambda_{a}I_{\alpha_{a}}+N_{\alpha_{a}} \\   
\end{pmatrix}   	$$
and
$$W_{B}= \begin{pmatrix} 
	\mu_{1}I_{\beta_{1}}+N_{\beta_{1}} &   &   &   \\  
	&	\mu_{2}I_{\beta_{2}}+N_{\beta_{2}}  &    &  \\       
	&   & \ddots  &  \\   
	&   &    &  \mu_{b}I_{\beta_{b}}+N_{\beta_{b}} \\   
\end{pmatrix}   	$$
Define
\[d_{ijk}=\sum_{r=2}^{k}\sum_{l=0}^{\min\{p_{i},q_{j}\}}\alpha_{i,r+l}(\beta_{j,1+l}-\beta_{j,2+k-r+l})+\sum_{l=0}^{\min\{p_{i},q_{j}\}}\alpha_{i,1+l}\sum_{s=1}^{k}\beta_{j,s+l}\]
where $\alpha_{i,l}=0$ if $l>p_{i}$ and $\beta_{j,l}=0$ if $l>q_{j}$.

Let $e_{ij}\in M_{m\times n}(\mathbb{C})$ be the matrix all of whose entries are zero except $(i, j)$ entry is 1.
Recall that $\gamma=(\gamma_{1}, \gamma_{2}, \dots , \gamma_{p})$ is said to be a partition of a positive integer $m$ with $p$ parts if $\gamma_{1}\geq \gamma_{2}\geq \dots\geq \gamma_{p}\geq 1$ are non-increasing sequences of positive integers whose sum $|\gamma|$ is $m$, write $\gamma\in P(m,p)$.
Let $\gamma\in P(m,p)$, $\delta\in P(n,q)$  and $\epsilon\in P(l,r)$.
Let $X$ be a $p\times q$ block matrix $(X_{ij})_{p\times q}$ of type $(\gamma,\delta)$, whose $(i,j)$ block is $\gamma_{i}\times \delta_{j}$ matrix $X_{ij}$.
That is to say, $\gamma$ and $\delta$ are respectively the partitions of the rows and columns of the matrix $X$.
Let $E_{i}(\gamma)$ be the $p\times p$ block diagonal matrix of type $\gamma$ all of whose blocks are zero except the $i$-th block is the identity matrix $I_{\gamma_{i}}$.
Then $E_{j}(\gamma)E_{k}(\gamma)= \delta_{jk}E_{j}(\gamma)$.
Assume $Z$ is a $k\times l$ matrix.
Let $E_{ij}(\gamma,\delta)\boxtimes Z$ be the $p\times q$ block matrix of type $(\gamma,\delta)$ all of whose blocks are zero except $(i,j)$ block is $\gamma_{i}\times \delta_{j}$ matrix whose upper left  block is  the upper left  block of $Z$ of $\min\{\gamma_{i},k\}\times \min\{\delta_{j},l\}$ and any other block is zero if exists.
Write $E_{ij}(\gamma)\boxtimes-$ for $E_{ij}(\gamma,\gamma)\boxtimes-$. 
We say $E_{ij}\boxtimes-$ is of type $(\gamma,\delta)$ if $E_{ij}\boxtimes-$ is $E_{ij}(\gamma,\delta)\boxtimes-$.
Then we have
$E_{ij}(\gamma,\delta)\boxtimes X_{ij}=E_{i}(\gamma)XE_{j}(\delta)$, and $X=\sum_{i,j} E_{ij}(\gamma,\delta)\boxtimes X_{ij}$.
If $Y$ is a $q\times r$ block matrix $(Y_{kl})_{q\times r}$ of type $(\delta, \epsilon)$, whose $(k,l)$ block $Y_{kl}$ is $ \delta_{k}\times \epsilon_{l}$ matrix, 
then 
$Y=\sum_{k,l} E_{kl}(\delta,\epsilon)\boxtimes Y_{kl}$, and $XY=\sum (E_{ij}(\gamma,\delta)\boxtimes X_{ij})(E_{kl}(\delta,\epsilon)\boxtimes Y_{kl})=\sum_{i,j}E_{ij}(\gamma,\epsilon)\boxtimes\sum_{k=1}^{q} X_{ik}Y_{kj}$, as we have $(E_{ij}(\gamma,\delta)\boxtimes X_{ij})(E_{kl}(\delta,\epsilon)\boxtimes Y_{kl})=\delta_{jk}E_{il}(\gamma,\epsilon)\boxtimes X_{ij}Y_{kl}$, where $\delta_{jk}$ is the Kronecker delta.

Note that  $W(A,\lambda_{i})=\lambda_{i}I_{\alpha_{i}}+N_{\alpha_{i}}$ and $N_{\alpha_{i}}=\sum_{j=1}^{p_{i}-1} E_{j,j+1}(\alpha_{i})\boxtimes N_{\alpha_{ij},\alpha_{i,j+1}}$ is the nilpotent Weyr matrix of Weyr characteristic $\alpha_{i}$. 
For the sake of brevity, the operator $\varphi_{N_{\alpha_{i}}N_{\beta_{j}}}$ will be abbreviated as  $\varphi_{\alpha_{i}\beta_{j}}$.

\begin{theorem}\label{main}
Suppose that $A\in M_{m\times m}(\mathbb{C})$, $B\in M_{n\times n}(\mathbb{C})$.
Let $\varphi_{AB}$ be the operator of $M_{m\times n}(\mathbb{C})$ defined by 
$$\varphi_{AB}(X)=AX-XB\quad  \text{for  }  X\in M_{m\times n}(\mathbb{C}).$$
Assume that $\alpha_{i}=(\alpha_{i1},\alpha_{i2},\dots ,\alpha_{ip_i})=\omega(A,\lambda_{i})$ is the Weyr characteristic of $A$ associated with eigenvalue $\lambda_{i}$, for $i=1,2,\dots, a$,  and $\beta_{j}=(\beta_{j1},\beta_{j2},\dots ,\beta_{jq_j})=\omega(B,\mu_{j})$ is the Weyr characteristic of $B$ associated with eigenvalue $\mu_{j}$, for $ j=1,2,\dots,b$.
Let $N_{\alpha_{i}}$ be the nilpotent Weyr matrix of Weyr characteristic $\alpha_{i}$, and 
denote $\varphi_{N_{\alpha_{i}}N_{\beta_{j}}}$ with the short form $\varphi_{\alpha_{i}\beta_{j}}$.
Then for each positive integer $k$,
$\operatorname{dim}\operatorname{ker}\varphi_{AB}^{k}=\sum_{i=1}^{a}\sum_{j=1}^{b}\delta_{\lambda_{i},\mu_{j}}\operatorname{dim} \operatorname{ker} \varphi_{\alpha_{i}\beta_{j}}^{k}$ is
\[
\sum_{i=1}^{a}\sum_{j=1}^{b}\delta_{\lambda_{i},\mu_{j}}\bigg(\sum_{r=2}^{k}\sum_{l=0}^{\min\{p_{i},q_{j}\}}\alpha_{i,r+l}(\beta_{j,1+l}-\beta_{j,2+k-r+l})+\sum_{l=0}^{\min\{p_{i},q_{j}\}}\alpha_{i,1+l}\sum_{s=1}^{k}\beta_{j,s+l}\bigg)
\]
where $\alpha_{i,l}=0$ if $l>p_{i}$ and $\beta_{j,l}=0$ if $l>q_{j}$.

A basis of $\operatorname{ker}\varphi_{W_{A}W_{B}}^{k}$ is
$$E_{ij}(\alpha, \beta)\boxtimes\left(\sum_{t=0}^{l}x_{kt}E_{r+t,s+t}(\alpha_{i},\beta_{j})\boxtimes Z_{r+l,s+l}(ij)\right)$$ 
where $(x_{k0}, x_{k1},\cdots,x_{kl})$ is the solution of the formula (\ref{formula}) in Lemma \ref{coefficient} for  $\varphi_{p_{i}q_{j}}^{k-1}$, and $Z_{r+l,s+l}(ij)$ runs over $\{e_{uv}\mid 1\leq u\leq \alpha_{i,r+l}, \beta_{j,1+k-r+s+l}+1\leq v\leq \beta_{j,s+l}\}$,
and $1\leq i\leq a$, $1\leq j\leq b$ such that $\lambda_{i}=\mu_{j}$, $2\leq r\leq \min\{p_{i},k\}$, $0\leq l\leq \min\{p_{i}-r,q_{j}-1\}$ when $s=1$, and $1\leq s\leq q_{j}$, $0\leq l\leq \min\{p_{i}-1,q_{j}-s\}$ when $r=1$.

Furthermore, the set of the eigenvalues of the operator $\varphi_{AB}$ is $\{\lambda_{i}-\mu_{j}\mid 1\leq i\leq a, 1\leq j\leq b\}$.
If the Weyr characteristic of the operator $\varphi_{AB}$ associated with an eigenvalue $\lambda$ is $\omega=(\omega_{1},\omega_{2},\cdots,\omega_{q})$,
then we have
\[
\omega_{k}=\sum_{i=1}^{a}\sum_{j=1}^{b}\delta_{\lambda_{i}-\lambda,\mu_{j}}\sum_{l=0}^{\min\{p_{i},q_{j}\}}\bigg(\alpha_{i,1+l}\beta_{j,k+l}+\sum_{r=2}^{k}\alpha_{i,r+l}(\beta_{j,1+k-r+l}-\beta_{j,2+k-r+l})\bigg)
\]
and the index $q$ of the eigenvalue $\lambda$ of $\varphi_{AB}$ equals $\max\{p_{i}+q_{j}-1\mid \lambda_{i}-\mu_{j}=\lambda\}$.
So the operator $\varphi_{AB}$ is characterised within similarity by $\operatorname{dim}\operatorname{ker}\varphi_{A-\lambda I, B}^{k}$, for $1\leq k\leq mn$, and for each eigenvalue $\lambda=\lambda_{i}-\mu_{j}$ of $\varphi_{AB}$.
\end{theorem}

First, we give a specific example.
\begin{example}
Let 
\[
A=\begin{pmatrix}
4&2&0& 2& 2 \\
 1& 3& 1& 0& 1 \\
 -2& -2& 1& -1& -2 \\
 -1& -2& -1& 1& 0\\ 
-1& -1& 0& -1& 1\\
\end{pmatrix}
\mbox{ and }
B=\begin{pmatrix}
5& 4& 2& 3& 1 \\ -1& 0& -1& -1& -1\\-2& -4& 0& -2& -2 \\-1& 0& 0& 1& 1 \\ 2& 4& 2& 2& 4 
 \end{pmatrix}
\]
Then each has $ \lambda=2$ as its sole eigenvalue.
\[
\begin{tabular}{|c|c|}
\hline
index of $A-2I$ & 3\\
\hline
$\operatorname{nullity}(A-2I)$&2
\\
$\operatorname{nullity}(A-2I)^{2}$&4\\
$\operatorname{nullity}(A-2I)^{3}$&5\\
\hline
Weyr structure of $A$\;&\;$(2,2,1)$\;\\
\hline
\end{tabular}
\qquad\quad
\begin{tabular}{|c|c|}
\hline
index of $B-2I$ & 2\\
\hline
$\operatorname{nullity}(B-2I)$&3\\
$\operatorname{nullity}(B-2I)^{2}$&5\\
\hline
Weyr structure of $B$\;&\;$(3,2)$\;\\
\hline
\end{tabular}
\]
Thus the Weyr characteristics of $A$ and $B$ are $(2,2,1)$ and $(3,2)$ respectively.
Hence, using the notations as in Theorem \ref{main}, we have $m=n=5$, $a=b=1$, $\lambda_{1}=\mu_{1}=2$, $\alpha_{1}=(2,2,1)$, $\beta_{1}=(3,2)$, $p_{1}=3$, $q_{1}=2$.
By Theorem \ref{main}, 
\[\operatorname{dim}\operatorname{ker}\varphi_{AB}^{k}=\sum_{r=2}^{k}\sum_{l=0}^{2}\alpha_{1,r+l}(\beta_{1,1+l}-\beta_{1,2+k-r+l})+\sum_{l=0}^{2}\alpha_{1,1+l}\sum_{s=1}^{k}\beta_{1,s+l}
\]
Thus 
\[\operatorname{dim}\operatorname{ker}\varphi_{AB}^{}=\sum_{l=0}^{2}\alpha_{1,1+l}\beta_{1,1+l}=\sum_{l=0}^{1}\alpha_{1,1+l}\beta_{1,1+l}=2\times 3+2\times 2=10\]
and
\[
\begin{split}
 \operatorname{dim}\operatorname{ker}\varphi_{AB}^{2}&=\sum_{l=0}^{2}\alpha_{1,2+l}(\beta_{1,1+l}-\beta_{1,2+l})+\sum_{l=0}^{2}\alpha_{1,1+l}\sum_{s=1}^{2}\beta_{1,s+l}\\
&=\sum_{l=0}^{1}\alpha_{1,2+l}(\beta_{1,1+l}-\beta_{1,2+l})+\sum_{l=0}^{1}\alpha_{1,1+l}\sum_{s=1}^{2}\beta_{1,s+l}\\
&=\alpha_{12}(\beta_{11}-\beta_{12})+\alpha_{13}(\beta_{12}-\beta_{13})+\alpha_{11}(\beta_{11}+\beta_{12})+\alpha_{12}(\beta_{12}+\beta_{13})\\
&=2(3-2)+1(2-0)+2(3+2)+2(2+0)\\
&=18
\end{split}
\]
Similarly, we have 
$\operatorname{dim}\operatorname{ker}\varphi_{AB}^{3}=23$, $\operatorname{dim}\operatorname{ker}\varphi_{AB}^{4}=25$.
Since $\lambda=0$ is the unique eigenvalue of the operator $\varphi_{AB}$, it follows that the Weyr characteristic of $\varphi_{AB}$ is $(10,8,5,2)$, which can also be calculated directly by the formula $\omega_{k}$ and $q$ in Theorem \ref{main}. 

A basis of $\operatorname{ker}\varphi_{W_{A}W_{B}}^{k}$ is given below, where $1\leq u\leq \alpha_{1,r+l}$, and
$1\leq v\leq \beta_{1,s+l}$ or $v=3$ when $k-r+s+l=1$.
\[
\begin{tabular}{|c|c|c|c|c|c|c|c|}
\hline
 \multirow{2}*{$(r,s,l)$} & \multirow{2}*{basis} & \multirow{2}*{$\alpha_{1,r+l}$}& \multirow{2}*{$ \beta_{1,s+l}$}&\multicolumn{4}{c|}{$k$}\\
\cline{5-8} 
&&&&1&2&3&4\\
\hline
(1,2,0)&$E_{12}\boxtimes e_{uv}$&2&2&&&&\\
\hline
(1,1,0)&$E_{11}\boxtimes e_{uv}$&2&3&$v=3$&&&\\
\hline
(1,1,1)&$(E_{11}+E_{22})\boxtimes e_{uv}$&2&2&&&&\\
\hline
(2,1,0)&$E_{21}\boxtimes e_{uv}$&2&3&---&$v=3$&&\\
\hline
(2,1,1)&$(E_{21}+2E_{32})\boxtimes e_{uv}$&1&2&---&&&\\
\hline
(3,1,0)&$E_{31}\boxtimes e_{uv}$&1&3&---&---&$v=3$&\quad\quad\quad \\
\hline
\end{tabular}
\]
For example, 
$e_{14},e_{15},e_{24},e_{25},e_{13},e_{23},e_{11}+e_{34},e_{12}+e_{35},e_{21}+e_{44},e_{22}+e_{45}$ form a basis of $\operatorname{ker}\varphi_{W_{A}W_{B}}$, and
$e_{14},e_{15},e_{24},e_{25},e_{11},e_{12},e_{13},e_{21},e_{22},e_{23},e_{11}+e_{34},e_{12}+e_{35},e_{21}+e_{44},e_{22}+e_{45}$, $e_{33}, e_{43}, e_{31}+2e_{54},e_{32}+2e_{55}$ form a basis of $\operatorname{ker}\varphi_{W_{A}W_{B}}^{2}$. Here each $e_{ij}$ is of size $5\times 5$.

Note that $E_{ij}\boxtimes e_{uv}$ in the table is the $3\times 2$ block matrix of type $(\alpha_{1},\beta_{1})$ all of whose blocks are zero except
$(i,j)$ block is $\alpha_{1i}\times \beta_{1j}$ matrix $e_{uv}$.
Here $(E_{21}+2E_{32})\boxtimes e_{uv}$ means  that $E_{21}\boxtimes e_{uv}+E_{32}\boxtimes 2e_{uv}$.
\qed
\end{example}
 
The main result in Theorem \ref{main} is exactly equivalent to giving the Weyr characteristic (Weyr structure)  for every eigenvalue of the operator $\varphi_{AB}$. 
Since knowing the Weyr structure of a nilpotent matrix is equivalent to knowing the nullities of its powers, see \cite[Proposition 2.2.3]{MCV}. So in effect the main result provides parameters that determine the operator within similarity. That is  to say, we obtain a complete set of similarity invariants of the operator.
More precisely, the operator $\varphi_{AB}$ is similar to $\varphi_{CD}$ if and only if 
$\operatorname{dim}\operatorname{ker}\varphi_{(A-\lambda I), B}^{k}=\operatorname{dim}\operatorname{ker}\varphi_{(C-\lambda I), D}^{k}$ for each eigenvalues $\lambda$ of  $\varphi_{AB}$ and $k \in\mathbb{N}$, where  $A, C\in M_{m\times m}(\mathbb{C})$, and $B, D\in M_{n\times n}(\mathbb{C})$. This should be of interest to a number of people including physicists.

\section{Preliminaries}
 
\begin{lemma}[{\cite[Lemma 2.1]{LLWX}}]\label{lemma1}
Let $\sigma(A)$ be the set of eigenvalues of a square matrix $A$. If $\sigma (A)\cap \sigma (B)= \emptyset$, then $\operatorname{ker} \varphi_{AB}^{\ell}=0$ for all $\ell\in \mathbb{N}$.
\end{lemma}

\begin{lemma}[{\cite[Lemma 2.2]{LLWX}}]\label{lemma2}
Let $\varphi$ be an operator of a finite dimensional vector space. Then the union of the preimage of a basis of $\operatorname{ker}\varphi\cap \operatorname{im}\varphi^{\ell}$ under $\varphi^{\ell}$ and a basis of $\operatorname{ker}\varphi^{\ell}$ form a basis of $\operatorname{ker}\varphi^{\ell+1}$ for all $\ell\in \mathbb{N}$. Consequently, 
$$\operatorname{dim}(\operatorname{ker}\varphi^{\ell+1})=\operatorname{dim}(\operatorname{ker}\varphi^{\ell})+\operatorname{dim}(\operatorname{ker}\varphi\cap \operatorname{im}\varphi^{\ell})$$
and
$$\operatorname{dim}(\operatorname{ker}\varphi^{\ell})=\sum_{i=0}^{\ell-1}\operatorname{dim}(\operatorname{ker}\varphi\cap \operatorname{im}\varphi^{i}).$$
\end{lemma}

Throughout the rest of this paper, we write $\gamma=(\gamma_{1}, \gamma_{2}, \dots , \gamma_{p})$ and $\delta=(\delta_{1}, \delta_{2}, \dots , \delta_{q})$ where $\gamma_{1}\geq \gamma_{2}\geq \dots\geq \gamma_{p}\geq 1$ and $\delta_{1}\geq \delta_{2}\geq \dots\geq \delta_{q}\geq 1$ are non-increasing sequences of positive integers. Let 
$$\Gamma=N_{\gamma}=\begin{pmatrix}
	0 & \Gamma_{12} &   &  & \\  
	& 0 & \Gamma_{23} &  &\\     
	&  & \ddots  &\ddots   &\\    
	&   &   &0 &\Gamma_{p-1,p} \\  
	&   &    & & 0 \\  
\end{pmatrix} 
\mbox{ and }
\Delta=N_{\delta}=\begin{pmatrix}
	0 & \Delta_{12} &   &  & \\  
	& 0 & \Delta_{23} &  &\\     
	&  & \ddots  &\ddots   &\\    
	&   &   &0 &\Delta_{q-1,q} \\  
	&   &    & & 0 \\  
\end{pmatrix} 
$$
where
$\Gamma_{ij} $ is the full-column-rank block $N_{\gamma_{i}\times \gamma_{j}}=\begin{pmatrix}
	I_{\gamma_{j}}  \\  
	0\\   
\end{pmatrix} $ of size $\gamma_{i}\times \gamma_{j}$ for $i< j$,
and $\Delta_{ij}=N_{\delta_{i}\times \delta_{j}}$.

Let $\varphi_{\gamma\delta}$ be the operator of $M_{m\times n}(\mathbb{C})$ defined by 
$$\varphi_{\gamma\delta}(X)=N_{\gamma} X-XN_{\delta}=\Gamma X-X\Delta\quad  \text{for  } X\in M_{m\times n}(\mathbb{C}).$$
Partition $X$ into $p\times q$ block matrix $(X_{i j})_{p\times q}$ of type $(\gamma,\delta)$, where the $(i,j)$ block  is $\gamma_{i}\times \delta_{j}$ matrix $X_{ij}$.  
Recall that $E_{ij}(\gamma,\delta)\boxtimes X_{ij}$ is the block matrix of type $(\gamma,\delta)$ all of whose blocks are zero except
$(i,j)$ block is $X_{ij}$. Then we have $X=\sum_{i,j} E_{ij}(\gamma,\delta)\boxtimes X_{ij}$.

\begin{lemma}\label{kernal-image}
Let $\varphi_{\gamma\delta}$ be the operator of $M_{m\times n}(\mathbb{C})$ defined by 
$$\varphi_{\gamma\delta}(X)=N_{\gamma} X-XN_{\delta}\quad  \text{for } X\in M_{m\times n}(\mathbb{C}).$$ where $\gamma=(\gamma_{1}, \gamma_{2}, \dots , \gamma_{p})$ and $\delta=(\delta_{1}, \delta_{2}, \dots , \delta_{q})$.
Then the formula for the dimension of $ \operatorname{ker}\varphi_{\gamma\delta}$ is
$$\operatorname{dim}(\operatorname{ker}\varphi_{\gamma\delta})=\sum_{i=1}^{\min \{p,q\}} \gamma_{i}\delta_{i}.$$
Moreover, 
$\sum_{l=0}^{q-j} E_{1+l, j+l}(\gamma,\delta)\boxtimes e_{kl}$ for $\max \{q-p,0 \}+1\leq j\leq q$, $1\leq k\leq \gamma_{q+1-j}$, $1\leq l\leq \delta_{q}$ and
 $\sum_{l=1}^{i} E_{l,j-i+l}(\gamma,\delta)\boxtimes e_{kl}$ for  $i\leq j< q$, $1\leq k\leq \gamma_{i}$, $\delta_{j+1}+1\leq l\leq \delta_{j}$ form a basis of $\operatorname{ker} \varphi_{\gamma\delta}$.
\end{lemma}
\begin{proof}
Note that  $\gamma$ and $\delta$ are partitions of $m$ and $n$, respectively.
So $X$ can be partitioned as $X=(X_{ij})_{p\times q}$ of type $(\gamma,\delta)$, where $X_{ij}$ is  matrix of $\gamma_{i}\times\delta_{j}$.
Then $X=\sum_{i,j} E_{ij}(\gamma,\delta)\boxtimes X_{ij}$.
Let $\Gamma=N_{\gamma}, \Delta=N_{\delta}$, and let $\Gamma$ and $\Delta$ be  similarly partitioned as $\Gamma=(\Gamma_{ij})_{p\times p}$ of type $\gamma$ and $\Delta=(\Delta_{ij})_{q\times q}$ of type $\delta$.
Then $\Gamma=\sum_{k=1}^{p-1} E_{k,k+1}(\gamma)\boxtimes \Gamma_{k,k+1}$ and 
$\Delta=\sum_{l=1}^{q-1} E_{l,l+1}(\delta)\boxtimes \Delta_{l,l+1}$.
Recall that
\[
\begin{split}
 \varphi_{\gamma\delta}(E_{ij}\boxtimes X_{ij})&= (\sum_{k=1}^{p-1} E_{k,k+1}\boxtimes \Gamma_{k,k+1})(E_{ij}\boxtimes X_{ij})-(E_{ij}\boxtimes X_{ij})(\sum_{l=1}^{q-1} E_{l,l+1}\boxtimes \Delta_{l,l+1})\\
& =E_{i-1,j}\boxtimes \Gamma_{i-1,i}X_{ij} -E_{i,j+1}\boxtimes X_{ij}\Delta_{j,j+1}
 \end{split}
 \]
Then 
\[
\begin{split}
 \varphi_{\gamma\delta}(X)&= (\sum_{k=1}^{p-1} E_{k,k+1}\boxtimes \Gamma_{k,k+1})(\sum_{i,j} E_{ij}\boxtimes X_{ij})-(\sum_{i,j} E_{ij}\boxtimes X_{ij})(\sum_{l=1}^{q-1} E_{l,l+1}\boxtimes \Delta_{l,l+1})\\
 &=\sum_{i,j} (E_{i-1,j}\boxtimes \Gamma_{i-1,i}X_{ij} -E_{i,j+1}\boxtimes X_{ij}\Delta_{j,j+1})\\
 &=\sum_{i,j} E_{ij}\boxtimes (\Gamma_{i,i+1}X_{i+1,j}-X_{i,j-1}\Delta_{j-1,j})
\end{split}
 \]
Hence, we have
\[
\begin{split}
 \varphi_{\gamma\delta}((E_{1,j-i}+\cdots+E_{i,j-1})\boxtimes X_{ij})&= -E_{ij}\boxtimes X_{ij} \; \mbox{if } j>i\\
\varphi_{\gamma\delta}((E_{i-j+2,1}+\cdots+E_{i,j-1})\boxtimes X_{ij})&= (E_{i-j+1,1}-E_{ij})\boxtimes X_{ij} \; \mbox{if } 2\leq j\leq i\\
\varphi_{\gamma\delta} (E_{i,j-1}\boxtimes Y_{i,j-1})&= E_{i-1,j-1}\boxtimes Y_{i,j-1} \; \mbox{if }  2\leq j\leq i\\
\end{split}
 \]
for $X_{ij}\in \operatorname{span}\{e_{kl}\mid 1\leq k\leq \gamma_{i}, 1\leq l\leq \delta_{j}\}$ and $Y_{ij}\in \operatorname{span}\{e_{kl}\mid 1\leq k\leq \gamma_{i}, \delta_{j+1}+1\leq l\leq \delta_{j}\}$.

Since $E_{ij}\boxtimes e_{kl}$, $j>i$, $1\leq k\leq \gamma_{i}$, $1\leq l\leq \delta_{j}$;
$(E_{i-j+1 , 1}-E_{ij})\boxtimes e_{kl}$, $2\leq j\leq i$, $1\leq k\leq \gamma_{i}$, $1\leq l\leq \delta_{j}$;
$E_{ij}\boxtimes e_{kl}$, $j\leq i<p$, $1\leq k\leq \gamma_{i+1}$, $\delta_{j+1}+1\leq l\leq \delta_{j}$ are linearly independent, so they form a basis of a subspace of $\operatorname{im} \varphi_{\gamma\delta}$ of dimension $\sum_{i\neq j}\gamma_{i}\delta_{j}$. 
Hence
$\operatorname{dim}(\operatorname{im}\varphi_{\gamma\delta})\geq\sum_{i\neq j}\gamma_{i}\delta_{j}$.
Since we have $\operatorname{dim}(\operatorname{im} \varphi_{\gamma\delta})+ \operatorname{dim}(\operatorname{ker}\varphi_{\gamma\delta})=mn$,
it follows that $\operatorname{dim}(\operatorname{ker}\varphi_{\gamma\delta} )\leq\sum_{i=1}^{\min \{p,q\}} \gamma_{i}\delta_{i}$.
Moreover, 
\[
\begin{split}
 \varphi_{\gamma\delta}(E_{1j}+E_{2 , j+1}+\dots +E_{q+1-j , q}\boxtimes X_{q+1-j , q})&=0 \quad \mbox{for }\max \{q-p,0 \}+1\leq j\leq q\\
\varphi_{\gamma\delta}(E_{1 , j-i+1}+E_{2 , j-i+2}+\dots +E_{ij}\boxtimes Y_{ij})&=0 \quad \mbox{for } i\leq j< q\
\end{split}
 \]
where $X_{ij}\in \operatorname{span}\{e_{kl}\mid 1\leq k\leq \gamma_{i}, 1\leq l\leq \delta_{j}\}$, and $Y_{ij}\in\operatorname{span}\{e_{kl}\mid 1\leq k\leq \gamma_{i}, \delta_{j+1}+1\leq l\leq \delta_{j}\}$. 

Since $(E_{1j}+E_{2 , j+1}+\dots +E_{q+1-j , q})\boxtimes e_{kl}$ for $\max \{q-p,0 \}+1\leq j\leq q$, $1\leq k\leq \gamma_{q+1-j}$, $1\leq l\leq \delta_{q}$;
$(E_{1 , j-i+1}+E_{2 , j-i+2}+\dots +E_{ij})\boxtimes e_{kl}$ for  $i\leq j< q$, $1\leq k\leq \gamma_{i}$, $\delta_{j+1}+1\leq l\leq \delta_{j}$
are linearly independent, so they form a basis of a subspace of $\operatorname{ker} \varphi_{\gamma\delta}$ of dimension $\sum_{i=1}^{\min \{p,q\}} \gamma_{i}\delta_{i}$. Hence they form  a basis of $\operatorname{ker}\varphi_{\gamma\delta}$ by comparing dimensions, where $E_{ij}\boxtimes-$ is of type $(\gamma,\delta)$.
\end{proof}

Belitskii was the first to observe the nice block upper triangular form of the matrices that centralize a given nilpotent matrix in Weyr form; see \cite[Proposition 3.2.2]{MCV}, for general matrix in Weyr form see \cite[Theorem 3.4.2.5 or Theorem 3.4.2.10]{HJ}.
Lemma \ref{kernal-image} recovers these results, and Corollary \ref{centralizer-dimension} gives the formula for the dimension of the centralizer space of a square matrix $A$.
This appears to have been first recorded by O’Meara and Vinsonhaler in \cite[Proposition 4.5 and Remark 4.6]{OV} and later in  \cite[Proposition 3.2.2]{MCV}.

\begin{corollary}\label{centralizer-dimension}
Suppose that $A\in M_{m\times m}(\mathbb{C})$, and 
$\alpha_{i}=(\alpha_{i1},\alpha_{i2},\dots ,\alpha_{ip_i})=\omega(A,\lambda_{i})$ is the Weyr characteristic of $A$ associated with $\lambda_{i}$ and $\lambda_{1},\lambda_{2},\dots,\lambda_{a}$ are distinct eigenvalues of $A$.
Then the formula for the dimension of the centralizer space of $A$ is 
\[
\sum_{i=1}^{a}\sum_{j=1}^{p_{i}} \alpha_{ij}^{2}.
\]
\end{corollary}

\begin{proof}
Let $\varphi_{A}$ be the operator of $M_{m\times m}(\mathbb{C})$ defined by 
$$\varphi_{A}(X)=AX-XA\quad  \text{for  }  X\in M_{m\times m}(\mathbb{C}).$$
Then $X$ is in the centralizer space of $A$  if and only if $X\in \operatorname{ker}\varphi_{A}$.
Let $X$ be partitioned into $a\times a$ blocks $X=(X_{ij})_{a\times a}$. 
Without loss of generality, we may assume that 
$$A= \begin{pmatrix}
	\lambda_{1}I_{\alpha_{1}}+N_{\alpha_{1}} &   &   &   \\   
	&	\lambda_{2}I_{\alpha_{2}}+N_{\alpha_{2}}  &    &  \\  
	&   & \ddots  &  \\  
	&   &    &  \lambda_{a}I_{\alpha_{a}}+N_{\alpha_{a}} \\   
\end{pmatrix}   	$$
For any $i\neq j$, since $\lambda_{i}\neq \lambda_{j}$, it follows from Lemma \ref{lemma1}, $(\lambda_{i}I_{\alpha_{i}}+N_{\alpha_{i}})X_{ij}-X_{ij}(\lambda_{j}I_{\alpha_{j}}+N_{\alpha_{j}} )=0$ if and only if $X_{ij}=0$.
For each $i$, $(\lambda_{i}I_{\alpha_{i}}+N_{\alpha_{i}})X_{ii}-X_{ii}(\lambda_{i}I_{\alpha_{i}}+N_{\alpha_{i}})=0$ is equivalent to $N_{\alpha_{i}}X_{ii}-X_{ii}N_{\alpha_{i}}=0$.
By Lemma \ref{kernal-image}, 
the formula for the dimension of $ \operatorname{ker}\varphi_{\alpha_{i}\alpha_{i}}$ is
$$\operatorname{dim}(\operatorname{ker}\varphi_{\alpha_{i}\alpha_{i}})=\sum_{j=1}^{p_{i}} \alpha_{ij}^{2}.$$
Hence the formula for the dimension of $ \operatorname{ker}\varphi_{A}$ is
$$\sum_{i=1}^{a}\operatorname{dim}(\operatorname{ker}\varphi_{\alpha_{i}\alpha_{i}})=\sum_{i=1}^{a}\sum_{j=1}^{p_{i}} \alpha_{ij}^{2}.$$
which is also the formula for the dimension of the centralizer space of $A$.
\end{proof}

\begin{lemma}[{\cite[Lemma 2.9]{LLWX}}] \label{invertible}
	Let $l$, $m$, $n$ be positive integers. Let $A=(a_{ij})_{1\leq i,j\leq l}$ be a square $l\times l$ matrix  with $(i,j)$ entry $a_{ij}$ is the binomial coefficient $\tbinom{m+n}{m+i-j}$. Then
	$$\det \bigg[ \bigg(\tbinom{m+n}{m+i-j}\bigg)_{1\leq i,j\leq l} \bigg]=\prod_{i=1}^{l}\dfrac{(m+n+i-1)!(i-1)!}{(m+i-1)!(n+i-1)!}
	$$
	In particular, $\det \bigg[ \bigg(\tbinom{m+n}{m+i-j}\bigg)_{1\leq i,j\leq l} \bigg]\neq 0.$
\end{lemma}

\begin{lemma}\label{coefficient}
Let $\varphi_{pq}$ be the operator of $M_{p\times q}(\mathbb{C})$ defined by
\[\varphi_{pq}(X)=N_pX-XN_q, \mbox{ for $X\in M_{p\times q}(\mathbb{C})$. }\]
where $N_p=e_{12}+e_{23}+\cdots+e_{p-1,p}$ and $N_q=e_{12}+e_{23}+\cdots+e_{q-1,q}$.
Then there is a unique solution of the formula
\begin{equation}
\label{formula}
\varphi_{pq}^{k}(\sum_{j=0}^{l}x_{j}e_{r+j,s+j})=\sum_{j=0}^{l}e_{1+j,k-r+1+s+j}+\sum_{j>l}y_{j}e_{1+j,k-r+1+s+j}
\end{equation}
for some $y_{j}$, where $r=1$ or $s=1$, $0\leq l\leq \min\{p-r, q-s\}$, $k+l\leq q+r-s-1$.
\end{lemma}

\begin{proof}
Note that $\varphi_{pq}^{k}(e_{ij})=\sum_{l=0}^{k}(-1)^{l}\tbinom{k}{l}e_{i-k+l,j+l}$. Then
\[
\begin{split}
\varphi_{pq}^{k}(\sum_{j=0}^{l}x_{j}e_{r+j,s+j})&=\sum_{j=0}^{l}x_{j}\sum_{t=0}^{k}(-1)^{t}\tbinom{k}{t}e_{r+j-k+t,s+j+t}\\
&=\sum_{u=0}^{k+l} \bigg(\sum_{j=0}^{l}(-1)^{u-j}\tbinom{k}{u-j}x_{j}\bigg)e_{r-k+u,s+u} \\
&=\sum_{u=k-r+1}^{k+l} \bigg(\sum_{j=0}^{l}(-1)^{u-j}\tbinom{k}{u-j}x_{j}\bigg)e_{r-k+u,s+u} 
\end{split}
\]
where $\tbinom{k}{l}=0$ if $l<0$ or $l>k$.
Let $m=k-r+1, n=r-1$, $u=m+i$. Then 
\[\varphi_{pq}^{k}(\sum_{j=0}^{l}x_{j}e_{r+j,s+j})=\sum_{i=0}^{l}\bigg(\sum_{j=0}^{l}(-1)^{m+i-j}\tbinom{m+n}{m+i-j}x_{j}\bigg)e_{1+i,s+m+i}+\sum_{i>l}y_{i}e_{1+i,s+m+i} \]
Thus the formula (\ref{formula})
\[\varphi_{pq}^{k}(\sum_{j=0}^{l}x_{j}e_{r+j,s+j})=\sum_{j=0}^{l}e_{1+j,k-r+1+s+j}+\sum_{j>l}y_{j}e_{1+j,k-r+1+s+j}\] 
is equivalent to the linear equations
\[\sum_{j=0}^{l}(-1)^{m+i-j}\tbinom{m+n}{m+i-j}x_{j}=1\quad \text{for } i=0,1,\cdots,l\]
By Lemma \ref{invertible}, the coefficient matrix is invertible, so there is a unique solution $(x_{0},x_{1},\cdots, x_{l})$ of the linear equations.
\end{proof}

\begin{lemma}\label{stratification}
Let $\varphi_{\gamma\delta}$ be the operator of $M_{m\times n}(\mathbb{C})$ defined by 
$$\varphi_{\gamma\delta}(X)=N_{\gamma} X-XN_{\delta}\quad  \text{for }  X\in M_{m\times n}(\mathbb{C}).$$ where $\gamma=(\gamma_{1}, \gamma_{2}, \dots , \gamma_{p})$ and $\delta=(\delta_{1}, \delta_{2}, \dots , \delta_{q})$.
Then  \[\varphi_{\gamma\delta}^{k}(\sum_{j=0}^{l}x_{j}E_{r+j,s+j}\boxtimes Z_{r+l,s+l})=\sum_{j=0}^{l}E_{1+j,k-r+1+s+j}\boxtimes Z_{r+l,s+l}\]
where $Z_{r+l,s+l}=\operatorname{span}\{e_{ij}\mid 1\leq i\leq \gamma_{r+l}, \delta_{2+k-r+s+l}+1\leq j\leq \delta_{s+l}\}$, and $(x_{0},x_{1},\cdots, x_{l})$ is the solution of the formula (\ref{formula}) in Lemma \ref{coefficient}. 
Furthermore, $$\varphi_{\gamma\delta}^{k+1}(\sum_{j=0}^{l}x_{j}E_{r+j,s+j}\boxtimes Z_{r+l,s+l})=0$$
where $E_{ij}\boxtimes-$ is of type $(\gamma,\delta)$.
\end{lemma}

\begin{proof}
Let $\Gamma=N_{\gamma}, \Delta=N_{\delta}$.
Note that
\[
 \varphi_{\gamma\delta}(E_{ij}\boxtimes X_{ij})
 =E_{i-1,j}\boxtimes \Gamma_{i-1,i}X_{ij} -E_{i,j+1}\boxtimes X_{ij}\Delta_{j,j+1}
 \]
Then we have
 \[
 \begin{split}
 \varphi_{\gamma\delta}^{k}(\sum_{j=0}^{l}x_{j}E_{r+j,s+j}\boxtimes Z_{r+l,s+l})
&=\sum_{j=0}^{l}E_{1+j,k-r+1+s+j}\boxtimes Z_{r+l,s+l}\\
&=\sum_{j=1}^{l+1}E_{j,k-r+s+j}\boxtimes Z_{r+l,s+l}\\
&=\sum_{t=1}^{l+1}(\sum_{j=1}^{t}E_{j,k-r+s+j})\boxtimes (Z_{r+l,s+l}\cap Y_{r+l, k-r+s+t})
\end{split}\]
where $Y_{rs}=\operatorname{span}\{e_{ij}\mid 1\leq i\leq \gamma_{r}, \delta_{s+1}+1\leq j\leq \delta_{s}\}$.
Since $\sum_{j=1}^{t}E_{j,k-r+s+j}\boxtimes (Z_{r+l,s+l}\cap Y_{r+l, k-r+s+t})\subset \operatorname{ker} \varphi_{\gamma\delta}$ by Lemma \ref{kernal-image}, 
it follows that  $\varphi_{\gamma\delta}^{k+1}(\sum_{j=0}^{l}x_{j}E_{r+j,s+j}\boxtimes Z_{r+l,s+l})=0$. Here $E_{ij}\boxtimes-$ is of type $(\gamma,\delta)$.
\end{proof}

\begin{lemma}\label{block}
Let $\varphi_{\gamma\delta}$ be the operator of $M_{m\times n}(\mathbb{C})$ defined by 
$$\varphi_{\gamma\delta}(X)=N_{\gamma} X-XN_{\delta}\quad  \text{for }  X\in M_{m\times n}(\mathbb{C}).$$ where $\gamma=(\gamma_{1}, \gamma_{2}, \dots , \gamma_{p})$ and $\delta=(\delta_{1}, \delta_{2}, \dots , \delta_{q})$.
Then  \[\operatorname{dim} \operatorname{ker} \varphi_{\gamma\delta}^{k}=\sum_{r=2}^{k}\sum_{l=0}^{\min\{p,q\}}\gamma_{r+l}(\delta_{1+l}-\delta_{2+k-r+l})+\sum_{l=0}^{\min\{p,q\}}\gamma_{1+l}\sum_{s=1}^{k}\delta_{s+l}\]
where $\gamma_{l}=0$ if $l>p$ and $\delta_{l}=0$ if $l>q$.
Moreover, a basis of $\operatorname{ker} \varphi_{\gamma\delta}^{k}$ is 
$$\sum_{j=0}^{l}x_{kj}E_{r+j,s+j}\boxtimes Z_{r+l,s+l}$$ where $(x_{k0},x_{k1},\cdots,x_{kl})$ is the solution of the formula (\ref{formula}) in Lemma \ref{coefficient} for $\varphi_{pq}^{k-1}$, and $Z_{r+l,s+l}$ runs over $\{e_{ij}\mid 1\leq i\leq \gamma_{r+l}, \delta_{1+k-r+s+l}+1\leq j\leq \delta_{s+l}\}$, and $2\leq r\leq k$, $0\leq l\leq \min\{p-r,q-1\}$ when $s=1$, and $1\leq s\leq q$, $0\leq l\leq \min\{p-1,q-s\}$ when $r=1$.
Here $E_{ij}\boxtimes-$ is of type $(\gamma,\delta)$.
\end{lemma}

\begin{proof}
By Lemma \ref{kernal-image} and Lemma \ref{stratification}, we have 
 \[ \begin{split}
 \operatorname{dim} \operatorname{ker} \varphi_{\gamma\delta}^{k}&=\sum_{r=2}^{k}\sum_{l=0}^{\min\{p,q\}}\gamma_{r+l}(\delta_{1+l}-\delta_{2+k-r+l})+\sum_{s=1}^{q}\sum_{l=0}^{\min\{p,q\}}\gamma_{1+l}(\delta_{s+l}-\delta_{k+s+l})\\
 &=\sum_{r=2}^{k}\sum_{l=0}^{\min\{p,q\}}\gamma_{r+l}(\delta_{1+l}-\delta_{2+k-r+l})+\sum_{l=0}^{\min\{p,q\}}\gamma_{1+l}\sum_{s=1}^{k}\delta_{s+l}\\
 &=\sum_{r=2}^{k}\sum_{l=0}^{\min\{p-r,q-1\}}\gamma_{r+l}(\delta_{1+l}-\delta_{2+k-r+l})+\sum_{l=0}^{\min\{p-1,q-s\}}\gamma_{1+l}\sum_{s=1}^{\min\{k,q-l\}}\delta_{s+l}\\
\end{split}\]
where $\gamma_{l}=0$ if $l>p$ and $\delta_{l}=0$ if $l>q$.

Furthermore, a basis of $\operatorname{ker} \varphi_{\gamma\delta}^{k}$ is 
$$\sum_{j=0}^{l}x_{kj}E_{r+j,s+j}\boxtimes Z_{r+l,s+l}$$ where $(x_{k0},x_{k1},\cdots,x_{kl})$  is the solution of the formula (\ref{formula}) in Lemma \ref{coefficient} for $\varphi_{pq}^{k-1}$, and $Z_{r+l,s+l}$ runs over $\{e_{ij}\mid 1\leq i\leq \gamma_{r+l}, \delta_{1+k-r+s+l}+1\leq j\leq \delta_{s+l}\}$, and $2\leq r\leq k$, $0\leq l\leq \min\{p-r,q-1\}$ when $s=1$, and $1\leq s\leq q$, $0\leq l\leq \min\{p-1,q-s\}$ when $r=1$.
Here $E_{ij}\boxtimes-$ is of type $(\gamma,\delta)$.

\end{proof}

\section{Proof of Theorem \ref{main}}\label{section3}
 
\begin{proof}
By the Weyr canonical form theorem \cite[Theorem 2.2.2]{MCV} or \cite[Theorem 3.4.2.3]{HJ}, 
every square matrix over an algebraically closed field is similar to a matrix in Weyr canonical form.
Hence there exist invertible matrices $P\in \operatorname{GL}_{m}(\mathbb{C})$ and $Q\in \operatorname{GL}_{n}(\mathbb{C})$
such that $P^{-1}AP=W_{A}$ and $Q^{-1}BQ=W_{B}$ are Weyr matrices, respectively.
Then for each $\ell$, we have $$P^{-1}\varphi_{AB}^{\ell}(X)Q=\varphi_{W_{A}W_{B}}^{\ell}(P^{-1}XQ)$$
Therefore, without loss of generality, we may assume that $A\in M_{m\times m}(\mathbb{C})$ and $B\in M_{n\times n}(\mathbb{C})$ are matrices in Weyr canonical form.

By hypothesis, $\alpha_{i}=(\alpha_{i1},\alpha_{i2},\dots ,\alpha_{ip_i})=\omega(A,\lambda_{i})$ is the Weyr characteristic of $A$ associated with eigenvalue $\lambda_{i}$, for $i=1,2,\dots, a$,  and $\beta_{j}=(\beta_{j1},\beta_{j2},\dots ,\beta_{jq_j})=\omega(B,\mu_{j})$ is the Weyr characteristic of $B$ associated with eigenvalue $\mu_{j}$, for $ j=1,2,\dots,b$.
Thus we may assume that
$$A= \begin{pmatrix}
	\lambda_{1}I_{\alpha_{1}}+N_{\alpha_{1}} &   &   &   \\   
	&	\lambda_{2}I_{\alpha_{2}}+N_{\alpha_{2}}  &    &  \\  
	&   & \ddots  &  \\  
	&   &    &  \lambda_{a}I_{\alpha_{a}}+N_{\alpha_{a}} \\   
\end{pmatrix}   	$$
and
$$B= \begin{pmatrix} 
	\mu_{1}I_{\beta_{1}}+N_{\beta_{1}} &   &   &   \\  
	&	\mu_{2}I_{\beta_{2}}+N_{\beta_{2}}  &    &  \\       
	&   & \ddots  &  \\   
	&   &    &  \mu_{b}I_{\beta_{b}}+N_{\beta_{b}} \\   
\end{pmatrix}   	$$

Let $X$ be partitioned into $a\times b$ blocks $X=(X_{ij})_{a\times b}$ of type $(\alpha,\beta)$,  where $(i,j)$ block matrix $X_{ij}$ is of size $|\alpha_{i}|\times |\beta_{j}|$, $|\alpha_{i}|=\sum_{k=1}^{p_{i}} \alpha_{ik}$, $|\beta_{j}|=\sum_{l=1}^{q_{j}} \beta_{jl}$, $\alpha=(|\alpha_{1}|,|\alpha_{2}|,\cdots,|\alpha_{a}|)$ and $\beta=(|\beta_{1}|,|\beta_{2}|,\cdots,|\beta_{b}|)$.
Partition $X_{ij}$ into $p_{i}\times q_{j}$ blocks $X_{ij}=(X_{i_{k}j_{l}})_{p_{i}\times q_{j}}$ of type $(\alpha_{i},\beta_{j})$,  where $(k,l)$ block matrix $X_{i_{k}j_{l}}$ is of size $\alpha_{ik}\times \beta_{jl}$.
Let $U=\varphi_{AB}(X)=(U_{ij})_{a\times b}$.
Then $U_{ij}=(\lambda_{i}I_{\alpha_{i}}+N_{\alpha_{i}})X_{ij}-X_{ij}(\mu_{j}I_{\beta_{j}}+N_{\beta_{j}})
=(\lambda_{i}-\mu_{j})X_{ij}+(N_{\alpha_{i}}X_{ij}-X_{ij}N_{\beta_{j}}).$
If $\lambda_{i}\neq \mu_{j}$, then by Lemma \ref{lemma1}, $U_{ij}=0$ if and only if $X_{ij}=0$.
If $\lambda_{i}=\mu_{j}$, then $U_{ij}=\varphi_{\alpha_{i}\beta_{j}}(X_{ij})$. For each $k\in\mathbb{N}$, we have $(\varphi_{AB}^{k}(X))_{ij}=\varphi_{\alpha_{i}\beta_{j}}^{k}(X_{ij})$ whenever $\lambda_{i}=\mu_{j}$. Therefore, we have
\[\operatorname{dim}\operatorname{ker}\varphi_{AB}^{k}=\sum_{i=1}^{a}\sum_{j=1}^{b} \delta_{\lambda_{i},\mu_{j}}\operatorname{dim} \operatorname{ker} \varphi_{\alpha_{i}\beta_{j}}^{k}\]
Let $d_{ijk}=\operatorname{dim} \operatorname{ker} \varphi_{\alpha_{i}\beta_{j}}^{k}$. Then by Lemma \ref{block}, we have 
\[d_{ijk}=\sum_{r=2}^{k}\sum_{l=0}^{\min\{p_{i},q_{j}\}}\alpha_{i,r+l}(\beta_{j,1+l}-\beta_{j,2+k-r+l})+\sum_{l=0}^{\min\{p_{i},q_{j}\}}\alpha_{i,1+l}\sum_{s=1}^{k}\beta_{j,s+l}\]
where $\alpha_{i,l}=0$ if $l>p_{i}$ and $\beta_{j,l}=0$ if $l>q_{j}$.
Thus  $\operatorname{dim}\operatorname{ker}\varphi_{AB}^{k}=\sum_{i=1}^{a}\sum_{j=1}^{b}\delta_{\lambda_{i},\mu_{j}}d_{ijk}$ is
\[
\sum_{i=1}^{a}\sum_{j=1}^{b}\delta_{\lambda_{i},\mu_{j}}\bigg(\sum_{r=2}^{k}\sum_{l=0}^{\min\{p_{i},q_{j}\}}\alpha_{i,r+l}(\beta_{j,1+l}-\beta_{j,2+k-r+l})+\sum_{l=0}^{\min\{p_{i},q_{j}\}}\alpha_{i,1+l}\sum_{s=1}^{k}\beta_{j,s+l}\bigg)
\]

A basis of $\operatorname{ker}\varphi_{AB}^{k}$ is
$$E_{ij}(\alpha, \beta)\boxtimes\left(\sum_{t=0}^{l}x_{kt}E_{r+t,s+t}(\alpha_{i},\beta_{j})\boxtimes Z_{r+l,s+l}(ij)\right)$$ 
where $(x_{k0},x_{k1},\cdots,x_{kl})$ is the solution of the formula (\ref{formula}) in Lemma \ref{coefficient} for $\varphi_{p_{i}q_{j}}^{k-1}$, and $Z_{r+l,s+l}(ij)$ runs over $\{e_{uv}\mid 1\leq u\leq \alpha_{i,r+l}, \beta_{j,1+k-r+s+l}+1\leq v\leq \beta_{j,s+l}\}$, and $1\leq i\leq a$, $1\leq j\leq b$ such that $\lambda_{i}=\mu_{j}$, $2\leq r\leq k$, $0\leq l\leq \min\{p_{i}-r,q_{j}-1\}$ when $s=1$, and $1\leq s\leq q_{j}$, $0\leq l\leq \min\{p_{i}-1,q_{j}-s\}$ when $r=1$.

By \cite[Proposition 2.2.5]{MCV}, an operator $\varphi_{AB}$ is determined within similarity by its eigenvalues and the Weyr characteristics (or Weyr structures) associated with these eigenvalues.
Let $\lambda$ be an eigenvalue of the operator $\varphi_{AB}$.
Let $\omega=(\omega_{1}, \omega_{2},\dots, \omega_{q})$ be the 
Weyr characteristic of the operator $\varphi_{AB}$ associated with $\lambda$.
Note that $\varphi_{AB}-\lambda I=\varphi_{(A-\lambda I), B}$,
so $\operatorname{nullity}(\varphi_{AB}-\lambda I)^{k}=\operatorname{dim}\operatorname{ker}\varphi_{(A-\lambda I),B}^{k}$.
Since we have $\omega_{1}=\operatorname{nullity}(\varphi_{AB}-\lambda I)$, and $\omega_{k}=\operatorname{nullity}(\varphi_{AB}-\lambda I)^{k}-\operatorname{nullity}(\varphi_{AB}-\lambda I)^{k-1}$, it follows that $\omega_{k}=\operatorname{dim}\operatorname{ker}\varphi_{(A-\lambda I),B}^{k}-\operatorname{dim}\operatorname{ker}\varphi_{(A-\lambda I),B}^{k-1}$ for $k \in\mathbb{N}$. Hence by \cite[Proposition 2.2.8]{MCV},
the operator $\varphi_{AB}$ is similar to $\varphi_{CD}$ if and only if 
$\operatorname{dim}\operatorname{ker}\varphi_{(A-\lambda I), B}^{k}=\operatorname{dim}\operatorname{ker}\varphi_{(C-\lambda I), D}^{k}$ for every eigenvalues $\lambda$ of  $\varphi_{AB}$ and $k \in\mathbb{N}$, where  $A, C\in M_{m\times m}(\mathbb{C})$, and $B, D\in M_{n\times n}(\mathbb{C})$.
Thus $\operatorname{dim}\operatorname{ker}\varphi_{(A-\lambda I), B}^{k}$ for each $k \in\mathbb{N}$ in Theorem \ref{main} is exactly equivalent to giving the Weyr characteristic (Weyr structure)  for each eigenvalue $\lambda$ of the operator $\varphi_{AB}$. 
So in effect the main result provides an invariant that determines the operator up to similarity.

We claim that the set of the eigenvalues of the operator $\varphi_{AB}$ is $\{\lambda_{i}-\mu_{j}\mid 1\leq i\leq a, 1\leq j\leq b\}$.
In fact, $\lambda$ is an eigenvalue of the operator $\varphi_{AB}$ if and only if $\operatorname{nullity}(\varphi_{AB}-\lambda I)>0$. Note that 
\[\operatorname{nullity}(\varphi_{AB}-\lambda I)=\operatorname{dim}\operatorname{ker}\varphi_{(A-\lambda I),B}^{}=
\sum_{i=1}^{a}\sum_{j=1}^{b}\delta_{\lambda_{i}-\lambda,\mu_{j}}\sum_{l=0}^{\min\{p_{i},q_{j}\}}\alpha_{i,1+l}\beta_{j,1+l}
\]
Hence $\operatorname{nullity}(\varphi_{AB}-\lambda I)>0$ if and only if $\lambda=\lambda_{i}-\mu_{j}$ for some $i$ and $j$.
Thus the set of eigenvalues of the operator $\varphi_{AB}$ is exactly the set $\{\lambda_{i}-\mu_{j}\mid 1\leq i\leq a, 1\leq j\leq b\}$ as claimed.

Let $\omega=(\omega_{1}, \omega_{2},\dots, \omega_{q})$ be the 
Weyr characteristic of the operator $\varphi_{AB}$ associated with an eigenvalue $\lambda$. Then we have 
\[
\begin{split}
\omega_{k}&=\operatorname{dim}\operatorname{ker}\varphi_{(A-\lambda I),B}^{k}-\operatorname{dim}\operatorname{ker}\varphi_{(A-\lambda I),B}^{k-1}\\
&=\sum_{i=1}^{a}\sum_{j=1}^{b}\delta_{\lambda_{i}-\lambda,\mu_{j}}\sum_{l=0}^{\min\{p_{i},q_{j}\}}\bigg(\alpha_{i,1+l}\beta_{j,k+l}+\sum_{r=2}^{k}\alpha_{i,r+l}(\beta_{j,1+k-r+l}-\beta_{j,2+k-r+l})\bigg)
\end{split}
\]
and the index $q$ of the eigenvalue $\lambda$ of $\varphi_{AB}$ is equal to  $\max\{p_{i}+q_{j}-1\mid \lambda_{i}-\mu_{j}=\lambda\}$.

Note that $\omega_{1}=\operatorname{nullity}(\varphi_{AB}-\lambda I)=\operatorname{dim}\operatorname{ker}\varphi_{(A-\lambda I),B}^{}$ is the geometric multiplicity of $\lambda$, and 
$\sum_{k=1}^{q}\omega_{k}=\operatorname{dim}\operatorname{ker}\varphi_{(A-\lambda I),B}^{q}$ is the algebraic multiplicity of $\lambda$.

\end{proof}

\textbf{Acknowledgments} The authors would like to thank Prof. Kevin O’Meara for his enthusiastic help and comments to improve this paper.

\end{document}